\documentclass{kms-j}

\newcommand{\publname}{J.~Korean Math.~Soc.}
\issueinfo{}% volume number
  {}%        % issue number
  {}%        % month
  {}%     % year
\pagespan{1}{}
\copyrightinfo{}%              % copyright year
  {Korean Mathematical Society}% copyright holder
\usepackage{graphicx}
\usepackage{color}
\allowdisplaybreaks

\theoremstyle{plain}
\newtheorem{theorem}{Theorem}[section]

\newtheorem{lemma}[theorem]{Lemma}
\newtheorem{corollary}[theorem]{Corollary}
\newtheorem{algorithm}[theorem]{Algorithm}
\theoremstyle{definition}

\theoremstyle{remark}

\begin{document}

\title[EQUILIBRIUM PROBLEMS, VARIATIONAL INEQUALITIES AND FIXED POINTS]
{A PARALLEL HYBRID METHOD FOR EQUILIBRIUM PROBLEMS, VARIATIONAL INEQUALITIES AND  NONEXPANSIVE MAPPINGS IN HILBERT SPACE}

\author[D. V. Hieu]{Dang Van Hieu}
\address{Dang Van Hieu \\ Department of Mathematics \\ Hanoi University of Science, Hanoi, Vietnam}
\email{dv.hieu83@gmail.com}
% 
% \author[H. Kim]{Hoil Kim}
% \address{Hoil Kim \\ Department of Mathematics \\ Kyungpook  National University \\ Taegu 702-701, Korea}
% \email{hikim@knu.ac.kr}
%\thanks{This work was financially supported by KRF 2003-041-C20009}

\subjclass{65Y05, 47H09, 47H10, 47J20}
\keywords{Hybrid method, equilibrium problem, variational inequality, parallel computation.}

\begin{abstract}
In this paper, a novel parallel hybrid iterative method is proposed for finding a common element of the set of solutions of a system of equilibrium problems, the set of solutions of variational inequalities for inverse strongly monotone mappings and the set of fixed points of a finite family of nonexpansive mappings in Hilbert space. Strong convergence theorem is proved for the sequence generated by the scheme. Finally, a parallel iterative algorithm for two finite families of variational inequalities and nonexpansive mappings is established.
\end{abstract}

\maketitle
\section{Introduction}
Let $H$ be a real Hilbert space with the inner product $\left\langle .,. \right\rangle$ and the norm $\left\|.\right\|$. Let $C$ be a nonempty closed convex subset of $H$. Let $A:C\to H$ be a (nonlinear) operator. The variational inequality problem is to find $p^*\in C$ such that
\begin{equation}\label{eq:VIP}
\left\langle Ap^*,p-p^* \right\rangle \ge 0,\quad \forall p\in C.
\end{equation}
The set of solutions of (\ref{eq:VIP}) is denoted by $VI(A,C)$.\\
A mapping $S:C\to C$ is said to be nonexpansive if $\left\|Sx-Sy\right\|\le \left\|x-y\right\|$ for all $x,y\in C$. The set of fixed points of $S$ is denoted by $F(S)=\left\{x\in C:S(x)=x\right\}$. 

For finding a common element of the set of fixed points of a nonexpansive mapping and the set of solutions of the variational inequality for an $\alpha$ - inverse strongly monotone mapping in Hilbert space, Takahashi and Toyoda \cite{TT2003} proposed the following iterative method: $x_0\in C$ and 
\begin{equation*}
x_{n+1}=\alpha_n x_n+(1-\alpha_n)SP_C(x_n-\lambda_n Ax_n)
\end{equation*}
for $n=0,1,2,\ldots$,  where $\lambda_n \in [a,b]$ for some $a, b \in(0,2\alpha)$ and $\alpha_n \in [c,d]$ for some $c,d\in(0,1)$. They proved that the sequence $\left\{x_n\right\}$ converges weakly to $z\in F(S)\cap VI(A, C)$, where $z=\lim_{n\to \infty}P_{F(S)\cap VI(A,C)}x_n$. To obtain strong convergence, Iiduka and Takahashi \cite{IT2004} proved the following convergence theorem:
\begin{theorem}\cite{IT2004}
Let $C$ be a closed convex subset of a real Hilbert space $H$. Let $A$ be an $\alpha$ - inverse - strongly - monotone mapping of $C$ into $H$ and let $S$ be a nonexpansive nonself-mapping of $C$ into $H$ such that $F(S) \cap VI(A,C)\ne\O$. Suppose $x_1=x\in C$ and $\left\{x_n\right\}$ is given by
\begin{equation*}
x_{n+1}=P_C\left(\alpha_n x_n+(1-\alpha_n)SP_C(x_n-\lambda_n Ax_n)\right),
\end{equation*}
for every $n=1,2,\ldots,$ where $\left\{\alpha_n\right\}$ is a sequence in $[0,1)$ and $\left\{\lambda_n\right\}$ is a sequence in $[0,2\alpha]$. If $\left\{\alpha_n\right\}$ and $\left\{\lambda_n\right\}$ are chosen so that $\lambda_n\in[a,b]$ for some $a,b$ with $0<a<b<2\alpha$,
\begin{equation*}
\lim_{n\to\infty}\alpha_n =0,\quad\sum_{n=1}^\infty\alpha_n=\infty,\quad\sum_{n=1}^\infty\left|\alpha_{n+1}-\alpha_n\right|<\infty,\quad\sum_{n=1}^\infty\left|\lambda_{n+1}-\lambda_n\right|<\infty.
\end{equation*}
Then $\left\{x_n\right\}$ converges strongly to $P_{F(S)\cap VI(A,C)}x$.
\end{theorem}
Let $f$ be a bifunction from $C\times C$ to the set of real numbers $\mathbb{R}$. The equilibrium problem for $f$ is to find an
element  $\widehat{x}\in C$, such that
\begin{equation}\label{eq:EP}
f(\widehat{x},y)\ge 0,\, \forall y\in C.
\end{equation} 
The set of solutions of the equilibrium problem $(\ref{eq:EP})$ is denoted by $EP(f)$. Equilibrium problems are generalized by several problems such as: optimization problems, variational inequalities, etc. In recent years, several methods have been proposed for finding a solution of equilibrium problem $(\ref{eq:EP})$ in Hilbert space \cite{BO1994,CH2005,TT2007,ZLL2011,YYL2012}.

In 2010, for finding a common element of the set of fixed points of nonexpansive mappings, the set of the solutions of  variational inequalities for $\alpha$-inverse strongly monotone operators, and the set of the solutions of equilibrium problems in Hilbert space, Saeidi \cite{S2010} proposed the following iterative method: $x_0\in H$ and
\begin{equation*}
\begin{cases}
&u_n=T_{r_{M,n}}^{f_M}\ldots T_{r_{1,n}}^{f_1}x_n,\\ 
&v_n=P_C(I-\lambda_{N,n}A_N)\ldots P_C(I-\lambda_{1,n}A_1)u_n,\\
&y_n=(1-\alpha_n)x_n+\alpha_n W_nv_n,\\
&C_n=\left\{v\in H:\left\|v-y_n\right\| \le \left\|v-x_n\right\|\right\},\\
&Q_n=\left\{v\in H: \left\langle x_0-x_n,x_n-v\right\rangle \ge 0\right\},\\
&x_{n+1}=P_{C_n\cap Q_n}x_0,n\ge 1,
\end{cases}
\end{equation*}
where, $W_n$ is the nonexpansive mapping, so-called the $W$-mapping \cite{T1997}, and $T_r^f x := u $ is the unique solution to the regularized equlibrium problem
$$f(u,y)+\frac{1}{r}\langle y-u,u-x\rangle\geq0, \quad \forall y\in C.$$
Clearly, Saeidi's algorithm is inherently sequential. Hence, when the numbers of operators $N$ and bifunctions $M$ are large, it is costly on a single processor.

Very recently, Anh and Chung \cite{AC2013} have proposed the following parallel hybrid iterative method for finding an element of the set of fixed points of a finite family of relatively nonexpansive mappings $\left\{S_i\right\}_{i=1}^N$: 
\begin{equation*}
\begin{cases}
&x_0\in C_0:=C, Q_0:=C,\\
&y_n^i=\alpha_n x_n+(1-\alpha_n)S_i x_n, i=1,\ldots, N,\\
&i_n:=\arg\max\left\{\left\|y_n^i -x_n \right\|:i=1,\ldots,N\right\}, \bar{y}_n:=y_n^{i_n},\\
&C_n=\left\{v\in C:\left\|v-\bar{y}_n\right\| \le \left\|v-x_n\right\|\right\},\\
&Q_n=\left\{v\in C : \left\langle x_0-x_n,x_n-v\right\rangle \ge 0\right\},\\
&x_{n+1}=P_{C_n\cap Q_n}x_0,n\ge 0.
\end{cases}
\end{equation*}
This algorithm was extended by Anh and Hieu \cite{AH2014} for a finite family of asymptotically quasi $\phi$ - nonexpansive mappings in Banach spaces.

In this paper, motivated by the results of  Takahashi et al \cite{IT2004,TT2003}, Saeidi \cite{S2010}, Anh and Chung \cite{AC2013}, we propose the following novel parallel hybrid iterative method for finding a common element of the set of solutions of a system of equilibrium problems for bifunctions $\left\{f_l\right\}_{l=1}^K$, the set of solutions of variational inequalies for $\alpha$-inverse strongly monotone mappings $\left\{A_k\right\}_{k=1}^M$ and the set of fixed points of a finite family of nonexpansive mappings $\left\{S_i\right\}_{i=1}^N$:
\begin{equation}\label{eq:Algorithm1}
\begin{cases}
&x_0\in H, C_0=Q_0=C,\\ 
&z_n^l=T_{r_n}^{f_l}x_n, l=1,\ldots, K,\\
&l_n:=\arg\max\left\{\left\|z_n^l -x_n\right\|:l=1,\ldots,K\right\}, \bar{z}_n:=z_n^{l_n},\\
&u_n^k=P_C (\bar{z}_n-\lambda A_k\bar{z}_n),k=1,\ldots, M,\\
& k_n:=\arg\max\left\{\left\|u_n^k -x_n \right\|:k=1,\ldots,M\right\}, \bar{u}_n:=u_n^{k_n},\\
&y_n^i=\alpha_n \bar{u}_n+(1-\alpha_n)S_i \bar{u}_n, i=1,\ldots, N,\\
& i_n:=\arg\max\left\{\left\|y_n^i -x_n \right\|:i=1,\ldots,N\right\}, \bar{y}_n:=y_n^{i_n},\\
 &C_n=\left\{v\in H:\left\|v-\bar{y}_n\right\| \le\left\|v-\bar{z}_n\right\|\le \left\|v-x_n\right\|\right\},\\
&Q_n=\left\{v\in H: \left\langle x_0-x_n,x_n-v\right\rangle \ge 0\right\},\\
&x_{n+1}=P_{C_n\cap Q_n}x_0,n\ge 0,
\end{cases}
\end{equation}
where $\lambda \in (0,2\alpha)$ and the control parameter sequences $\left\{\alpha_n\right\},\left\{r_n\right\}$ satisfy some conditions. Clearly, in the method $(\ref{eq:Algorithm1})$, at $n^{th}$ step, we can calculate the intermediate approximations $z_n^l$ in parallel. Then, among all $z_n^l$, the element $\bar{z}_n$ which is farest from $x_n$ is selected. Using the element $\bar{z}_n$ to find the approximations $u_n^k$ in parallel. After that, we chose the element $\bar{u}_n$ that is farest from $x_n$ among $u_n^k$. Similarly, $y_n^i$ are calculated in parallel and $\bar{y}_n$ is determined. Based on $\bar{y}_n$, $\bar{z}_n$, $x_n$, the closed and convex subsets $C_n, Q_n$ are constructed. Finally, the next approximation $x_{n+1}$ is determined as the projection of $x_0$ onto the intersection $C_n\cap Q_n$ of two closed and convex subsets $C_n$ and $Q_n$.

This paper is organized as follows: In Section 2, we collect some definitions and results for researching into the convergence of the proposed method. Section 3 deals with the convergence analysis of the method and its applications.
\section{Preliminaries}
In what follows, we review some definitions and results, which are employed in this paper. We refer the reader to \cite{IT2004}. We write $x_n\to x$ to indicate that the sequence $\left\{x_n\right\}$ converges strongly to $x$ and $x\rightharpoonup x$ implies that $\left\{x_n\right\}$ converges weakly to $x$. 

A mapping $A:C\to H$ is called \textit{$\alpha$ - inverse strongly monotone} if there exists a constant $\alpha>0$ such that
\begin{equation*}\label{eq:2.30}
\left\langle Ax-Ay,x-y \right\rangle \ge \alpha\left\|Ax-Ay\right\|^2
\end{equation*}
for all $x,y\in C$ and \textit{$\eta$ - strongly monotone} if there exists $\eta>0$ such that 
\begin{equation*}\label{eq:2.30*}
\left\langle Ax-Ay,x-y \right\rangle \ge \eta\left\|x-y\right\|^2.
\end{equation*}
It is well known that if $A$ is $\eta$ - strongly monotone and $L$ - Lipschitz, i.e., $\left\|Ax-Ay\right\|\le L\left\|x-y\right\|$ for all $x,y\in C$ then $A$ is $\eta/L^2$ -  inverse strongly monotone. If $A:C\to H$ is $\alpha$ - inverse strongly monotone then $A$ is $1/\alpha$ - Lipschitz continuous and $I-\lambda A$ is nonexpansive of $C$ onto $H$, where $\lambda\in (0,2\alpha)$. If $T$ is nonexpansive then $A=I-T$ is $1/2$ - inverse strongly monotone and $VI(A,C)=F(T)$. 

For every $x\in H$, the element $P_C x$ is defined by
\begin{equation*}\label{eq:1.3}
P_C x=\arg\min\left\{\left\|y-x\right\|:y\in C\right\}.
\end{equation*}
Since C is a nonempty closed and convex subset of $H$, $P_C x$ is existent and unique. Mapping $P_C:H\to C$ is called the projection of $H$ onto $C$. It is also known that $P_C$ satisfies
\begin{equation}\label{eq:FirmlyNonexpOfPC}
\left\langle P_C x-P_C y,x-y \right\rangle \ge \left\|P_C x-P_C y\right\|^2.
\end{equation}
This implies that $P_C$ is $1$ - inverse strongly monotone and for all $x\in C, y\in H$, we have
\begin{equation}\label{eq:ProperOfPC}
\left\|x-P_C y\right\|^2+\left\|P_C y-y\right\|^2\le \left\|x-y\right\|^2.
\end{equation}
Moreover, $z=P_C x$ if only if 
\begin{equation}\label{eq:EquivalentPC}
\left\langle x-z,z-y \right\rangle \ge 0,\quad \forall y\in C,
\end{equation}
and this implies that $p^*\in VI(A,C)$ if only if 
\begin{equation}\label{eq:EquivalentVIP}
p^*=P_C (p^*-\lambda Ap^*), \quad \lambda >0.
\end{equation} 
We have the following result of the convexity and closedness of $VI(A,C)$.
\begin{lemma}\label{lem:T2000}\cite{T2000}
 Let C be a nonempty, closed convex subset of a Banach space E and A be a monotone, hemicontinuous operator of C into $E^*$. Then
\begin{equation*}
VI(A,C)=\left\{u\in C:\left\langle v-u,Av\right\rangle\ge 0, \quad for\,all\,\,\, v\in C\right\}.
\end{equation*}
\end{lemma}
Next, for solving the equilibrium problem $(\ref{eq:EP})$, we assume that the bifunction $f$ satisfies the following conditions:
\begin{enumerate}
\item[(A1)] $f(x,x)=0$ for all $x\in C$;
\item[(A2)] f is monotone, i.e., $f(x,y)+f(y,x)\le 0$ for all $x, y \in C$;
\item[(A3)] For all $x,y,z \in C$,
$$ \lim_{t\to 0^+}\sup f(tz+(1-t)x,y) \le f(x,y); $$
\item[(A4)] For all $x\in C$, $f(x,.)$ is convex and lower semicontinuous.
\end{enumerate}
The following results concern with the bifunction $f$ :
\begin{lemma}\label{ExitenceN0} \cite{CH2005} Let $C$ be a
closed and convex subset of Hilbert space H, $f$ be a bifunction from $C\times C$ to
$\mathbb{R}$ satisfying the conditions $(A1)$-$(A4)$ and let $r>0$,
$x\in H$. Then, there exists $z\in C$ such that
\begin{eqnarray*}
f(z,y)+\frac{1}{r}\langle y-z,z-x\rangle\geq0, \quad \forall y\in C.
\end{eqnarray*}
\end{lemma}
\begin{lemma}\label{Tr-ClosedConvex}\cite{CH2005} Let
$C$ be a closed and convex subset of a Hilbert space $H$, $f$ be a bifunction from
$C\times C$ to $\mathbb{R}$ satisfying the conditions $(A1)$-$(A4)$.
For all $r>0$ and $x\in H$, define the mapping
\begin{eqnarray*}
T_r^f x=\{z\in C:f(z,y)+\frac{1}{r}\langle y-z,z-x\rangle\geq0, \quad
\forall y\in C\}.
\end{eqnarray*}
Then the following hold:

{\rm (B1)} $T_r^f$ is single-valued;

{\rm (B2)} $T_r^f$ is a firmly nonexpansive, i.e., for
all $x, y\in H,$
\begin{eqnarray*}
||T_r^fx-T_r^fy||^2\leq\langle T_r^fx-T_r^fy,x-y\rangle;
\end{eqnarray*}

{\rm (B3)} $F(T_r^f)=EP(f);$

{\rm (B4)} $EP(f)$ is closed and convex.
\end{lemma}
\begin{lemma}\cite{GK1990}\label{lem.demiclose}
Assume that $T:C\to C$ is a nonexpansive mapping. If $T$ has a fixed point , then 
\begin{enumerate}
\item [$(i)$] $F(T)$ is closed convex subset of $H$.
\item [$(ii)$] $I-T$ is demiclosed, i.e., whenever $\left\{x_n\right\}$ is a sequence in $C$ weakly converging to some $x\in C$ and the sequence $\left\{(I-T)x_n\right\}$ strongly converges to some $y$ , it follows that $(I-T)x=y$.
\end{enumerate}
\end{lemma}
\section{Main results}
In this section, we shall prove the convergence theorem for the method $(\ref{eq:Algorithm1})$. Putting
$$F=\left(\cap_{l=1}^K EP(f_l)\right)\bigcap\left(\cap_{i=1}^N F(S_i)\right)\bigcap \left(\cap_{k=1}^M VI(A_k,C)\right)$$
and assume that $F$ is the nonempty set. We also propose a simplier algorithm than the algorithm $(\ref{eq:Algorithm1})$ for a system of variational inequalities and a finite family of nonexpansive mappings.
\begin{theorem}\label{theo:VIPandRNMs1}
Let $\left\{A_k\right\}_{k=1}^M:C\to H$ be a finite family of $\alpha$ - inverse strongly monotone operators, $\left\{S_i\right\}_{i=1}^N:C\to C$ be a finite family of nonexpansive mappings, and $\left\{f_l\right\}_{l=1}^K$ be a finite family of befunctions from $C\times C$ to $\mathbb{R}$ satisfing the conditions $(A1)-(A4)$. Assume that the set $F$ is nonempty, $\lambda\in(0;2\alpha)$ and the control parameter sequences $\left\{\alpha_n\right\}$ and $\left\{r_n\right\}$ satisfy the following conditions:
\begin{itemize}
\item [$(i)$] $\left\{\alpha_n\right\}\subset [0,1], \lim\sup_{n\to \infty}\alpha_n <1$;
\item [$(ii)$] $\left\{r_n\right\}\subset [d,\infty)$ for some $d>0$.
\end{itemize}
Then the sequence $\left\{x_n\right\}$ is generated by algorithm $(\ref{eq:Algorithm1})$ converges strongly to $P_{F}x_0$.
\end{theorem}
\begin{proof}
We divide the proof of Theorem $\ref{theo:VIPandRNMs1}$ into seven steps.\\
\textbf{Step 1.} We show that $F, C_n, Q_n$ are closed convex subsets of $H$. By Lemmas $\ref{lem:T2000}$, $\ref{Tr-ClosedConvex}$, and $\ref{lem.demiclose}$, $EP(f_l)$, $VI(A_k,C)$, $F(S_i)$ are closed and convex. Hence, $F$ is closed and convex. From the definitions of $C_n, Q_n$, we see that $Q_n$ is closed and convex and $C_n$ is closed. Now, we show that $C_n$ is convex. Indeed, the inequality $\left\|v-\bar{y}_n\right\| \le  \left\|v-x_n\right\|$ is equivalent to
\begin{equation*}
\left\langle v,x_n-\bar{y}_n\right\rangle \le \frac{1}{2}\left(\left\|x_n \right\|^2-\left\|\bar{y}_n \right\|^2\right).
\end{equation*}
This implies that $C_n$ is convex for all $n\ge 0$, and so $\Pi_{C_n\cap Q_n}x_0$ and $\Pi_Fx_0$ are well-defined.\\
\textbf{Step 2.} We show that $F \subset C_n \cap Q_n$ for all $n\ge 0$. We have $y_n^i=\alpha_n x_n-(1-\alpha_n)S_{i} \bar{u}_n$. For every $u\in F$, by the convexity of $\left\|.\right\|^2$ and the nonexpansiveness of $S_{i_n}$, we obtain
\begin{align}
\left\|u-\bar{y}_n\right\|^2=&\left\|u-\alpha_n \bar{u}_n-(1-\alpha_n)S_{i_n} \bar{u}_n\right\|^2\notag \\
=& \left\|u\right\|^2-2\alpha_n\left\langle u,\bar{u}_n\right\rangle -2(1-\alpha_n)\left\langle u,S_{i_n} \bar{u}_n\right\rangle\notag\\
&+\left\|\alpha_n x_n+(1-\alpha_n)S_{i_n}\bar{u}_n\right\|^2\notag\\
\le& \left\|u\right\|^2-2\alpha_n\left\langle u,\bar{u}_n\right\rangle -2(1-\alpha_n)\left\langle u,S_{i_n} \bar{u}_n\right\rangle+\alpha_n\left\|x_n\right\|^2\notag\\
&+(1-\alpha_n)\left\|S_{i_n}\bar{u}_n\right\|^2 \notag\\
=&\alpha_n\left\|u-\bar{u}_n\right\|^2+(1-\alpha_n)\left\|u-S_{i_n} \bar{u}_n\right\|^2\notag\\
\le& \alpha_n\left\|u-\bar{u}_n\right\|^2+(1-\alpha_n)\left\|u-\bar{u}_n\right\|^2\notag\\
= &\left\|u-\bar{u}_n\right\|^2.\label{eq:2.35}
\end{align}
From $(\ref{eq:FirmlyNonexpOfPC})$, the definition of $\bar{u}_n$, and the nonexpansiveness of $P_C(I-\lambda A_{k_n}), T_{r_n}^{f_l}$, we have
\begin{eqnarray}
\left\|u-\bar{u}_n\right\|&=&\left\|P_C (I-\lambda A_{k_n})u-P_C (I-\lambda A_{k_n})\bar{z}_n\right\|\notag\\
&\le& \left\|u-\bar{z}_n\right\|\notag\\
&=&||T_{r_n}^{f_{l_n}}u-T_{r_n}^{f_{l_n}}x_n||\notag\\
&\le&||u-x_n||.\label{eq:2.36}
\end{eqnarray}
From $(\ref{eq:2.35})$ and $(\ref{eq:2.36}),$
\begin{equation}\label{eq:2.37}
\left\|u-\bar{y}_n\right\|\le\left\|u-\bar{z}_n\right\|\le\left\|u-x_n\right\|.
\end{equation}
This implies that $F \subset C_n$ for all $n\ge 0$. Next, we show that $F \subset C_n \cap Q_n$ for all $n\ge 0$ by the induction. Indeed, we have that $C_0=Q_0=C$ and $F \subset C=C_0 \cap Q_0$. Assume that $F\subset C_n \cap Q_n$ for some $n\ge 0$. From $x_{n+1}=P_{C_n\cap Q_n}x_0$ and  $(\ref{eq:EquivalentPC})$, we get
\begin{equation*}
\left\langle x_{n+1}-z,x_0 -x_{n+1}\right\rangle \ge 0
\end{equation*}
for all $z\in C_n \cap Q_n$. Since $F \subset C_n \cap Q_n$, $\left\langle x_{n+1}-z,x_0 -x_{n+1}\right\rangle \ge 0$ for all $z\in F$. This together with the definition of $Q_{n+1}$ implies that $F \subset Q_{n+1}$. Hence $F \subset C_n \cap Q_n$ for all $n\ge 0$.\\
\textbf{Step 3.} We show that $\left\|x_{n}-y_n^i\right\|\to 0$ and $\left\|x_{n}-z_n^l\right\|\to 0$ as $n\to \infty$ for all $i=1,2,\ldots N,l=1,2,\ldots K$. From the definition of $Q_n$ and $(\ref{eq:EquivalentPC})$, we see that  $x_n=P_{Q_n}x_0$. Therefore, for every $u\in F\subset Q_n$, we get
\begin{equation}\label{eq:2.37*}
\left\|x_n-x_0 \right\|^2\le \left\|u-x_0 \right\|^2-\left\|u-x_n \right\|^2\le\left\|u-x_0 \right\|^2.
\end{equation}
This implies that the sequence $\left\{x_n\right\}$ is bounded. From $(\ref{eq:2.36})$, $\left\{u_n^k\right\}$ is bounded.  By the nonexpansiveness of $S_i$, the sequence $\left\{S_i u_n^k\right\},\left\{y_n^i\right\}$ are also bounded.

We have $x_{n+1}=P_{C_n\cap Q_n}x_0\in Q_n, x_n=P_{Q_n}x_0$, from $(\ref{eq:ProperOfPC})$ we get 
\begin{equation}\label{eq:2.38}
\left\|x_n-x_0 \right\|^2\le \left\|x_{n+1}-x_0 \right\|^2-\left\|x_{n+1}-x_n \right\|^2\le \left\|x_{n+1}-x_0 \right\|^2.
\end{equation}
Hence the sequence $\left\{\left\|x_n-x_0 \right\|\right\}$ is nondecreasing, and so there exists the limit of the sequence $\left\{\left\|x_n-x_0 \right\|\right\}$. From $(\ref{eq:2.38})$ we obtain
\begin{equation*}\label{eq:2.39}
\left\|x_{n+1}-x_n \right\|^2\le \left\|x_{n+1}-x_0 \right\|^2-\left\|x_n-x_0 \right\|^2.
\end{equation*}
Taking $n\to\infty$, we obtain
\begin{equation}\label{eq:2.40}
\lim_{n\to\infty}\left\|x_{n+1}-x_n \right\|=0.
\end{equation}
From $x_{n+1}=P_{C_n\cap Q_n}x_0\in C_n$ and the definition of $C_n$, we have that
\begin{equation*}\label{eq:2.41}
\left\|x_{n+1}-\bar{y}_n \right\|\le\left\|x_{n+1}-\bar{z}_n\right\|\le\left\|x_{n+1}-x_n \right\|.
\end{equation*}
Therefore,
\begin{equation}\label{eq:2.42}
\lim_{n\to\infty}\left\|x_{n+1}-\bar{y}_n \right\|=\lim_{n\to\infty}\left\|x_{n+1}-\bar{z}_n \right\|=0.
\end{equation}
By $(\ref{eq:2.40}),(\ref{eq:2.42})$ and the estimate $||x_n-\bar{y}_n||\le||x_n-x_{n+1}||+||x_{n+1}-\bar{y}_n||$, we get
\begin{equation*}\label{eq:2.43}
\lim_{n\to\infty}\left\|x_{n}-\bar{y}_n\right\|=0.
\end{equation*}
From the definition of $i_n$, we obtain
\begin{equation}\label{eq:2.44}
\lim_{n\to\infty}\left\|x_{n}-y_n^i\right\|=0
\end{equation}
for all $i=1,2,\ldots,N$. By arguing similarly to $(\ref{eq:2.44})$, we obtain
\begin{equation}\label{eq:2.44t}
\lim_{n\to\infty}\left\|x_{n}-z_n^l\right\|=0, l=1,2,\ldots,K.
\end{equation}\\
\textbf{Step 4.} We show that $\lim_{n\to\infty}\left\|x_{n}-S_i x_n\right\|=0$. From $y_n^i=\alpha_n x_n+(1-\alpha_n)S_i \bar{u}_n$, we obtain
\begin{align*}
\left\|x_{n}-y_n^i\right\|=(1-\alpha_n)\left\|x_{n}-S_i \bar{u}_n\right\|.
\end{align*}
Therefore,
\begin{align}
\left\|x_{n}-S_i x_n\right\|&\le \left\|x_{n}-S_i \bar{u}_n\right\|+\left\|S_i \bar{u}_n-S_i x_n\right\|\notag\\
&\le \left\|x_{n}-S_i \bar{u}_n\right\|+\left\|\bar{u}_n-x_n\right\|\notag\\
&=\frac{1}{1-\alpha_n}\left\|x_{n}-y_n^i\right\|+\left\|\bar{u}_n-x_n\right\|.\label{eq:2.45}
\end{align}
For every $u\in F$, from $(\ref{eq:FirmlyNonexpOfPC})$ and $(\ref{eq:EquivalentPC})$, we see that 
\begin{align*}
2\left\|u-\bar{u}_n\right\|^2&=2\left\|P_C (u-\lambda A_{k_n}u)-P_C (\bar{z}_n-\lambda A_{k_n}\bar{z}_n)\right\|^2\\
&\le 2\left\langle (u-\lambda A_{k_n}u)-(\bar{z}_n-\lambda A_{k_n}\bar{z}_n), u-\bar{u}_n \right\rangle\\
&=||(u-\lambda A_{k_n}u)-(\bar{z}_n-\lambda A_{k_n}\bar{z}_n)||^2+||u-\bar{u}_n||^2\\
&\quad-||(u-\lambda A_{k_n}u)-(\bar{z}_n-\lambda A_{k_n}\bar{z}_n)-\left(u-\bar{u}_n\right)||^2\\
&\le ||u-\bar{z}_n||^2+||u-\bar{u}_n||^2-||\left(\bar{z}_n-\bar{u}_n\right)-\lambda(A_{k_n}\bar{z}_n-A_{k_n}u)||^2\\
&= ||u-\bar{z}_n||^2+||u-\bar{u}_n||^2-||\bar{u}_n-\bar{z}_n||^2-\lambda^2||A_{k_n}\bar{z}_n-A_{k_n}u||^2\\
&\quad +2\lambda\left\langle \bar{z}_n-\bar{u}_n,A_{k_n}\bar{z}_n-A_{k_n}u\right\rangle.
\end{align*}
Therefore,
\begin{align}
\left\|u-\bar{u}_n\right\|^2&\le||u-\bar{z}_n||^2-||\bar{u}_n-\bar{z}_n||^2+2\lambda\left\langle \bar{z}_n-\bar{u}_n,A_{k_n}\bar{z}_n-A_{k_n}u\right\rangle \notag\\ 
&\le\left(||u-\bar{z}_n||^2-||\bar{u}_n-\bar{z}_n||^2\right)+2\lambda||\bar{u}_n-\bar{z}_n||||A_{k_n}\bar{z}_n-A_{k_n}u||\notag\\
&\le\left(||u-x_n||^2-||\bar{u}_n-\bar{z}_n||^2\right)+2\lambda||\bar{u}_n-\bar{z}_n||||A_{k_n}\bar{z}_n-A_{k_n}u||.\label{eq:2.47*}
\end{align}
From the convexity of $\left\|.\right\|^2$ and the nonexpansiveness of $S_i$ we have
\begin{align}
\left\|u-y_n^i\right\|^2&=\left\|u-(\alpha_n \bar{u}_n+(1-\alpha_n)S_i \bar{u}_n)\right\|^2\notag\\ 
& \le \alpha_n \left\|u-\bar{u}_n\right\|^2+(1-\alpha_n)\left\|u-S_i \bar{u}_n\right\|^2\notag\\
&\le \alpha_n \left\|u-\bar{u}_n\right\|^2+(1-\alpha_n)\left\|u-\bar{u}_n\right\|^2\notag\\
&=\left\|u-\bar{u}_n\right\|^2\notag\\
&=\left\|P_C(u-\lambda A_{k_n}u)-P_C(\bar{z}_n-\lambda A_{k_n}\bar{z}_n)\right\|^2\notag\\
&\le \left\|(u-\lambda A_{k_n}u)-(\bar{z}_n-\lambda A_{k_n}\bar{z}_n)\right\|^2\notag\\
&= \left\|\lambda(A_{k_n}\bar{z}_n-A_{k_n}u)-(\bar{z}_n-u)\right\|^2\notag\\
&= \lambda^2\left\|A_{k_n}\bar{z}_n-A_{k_n}u\right\|^2-2\lambda\left\langle A_{k_n}\bar{z}_n-A_{k_n}u,\bar{z}_n-u\right\rangle+||\bar{z}_n-u||^2\notag\\
&\le \left\|u-x_n\right\|^2-\lambda(2\alpha-\lambda)\left\|A_{k_n}\bar{z}_n-A_{k_n}u\right\|^2.\label{eq:2.49}
\end{align}
This implies that 
\begin{equation}\label{eq:2.50}
\lambda(2\alpha-\lambda)\left\|A_{k_n}\bar{z}_n-A_{k_n}u\right\|^2\le \left\|u-x_n\right\|^2-\left\|u-y_n^i\right\|^2.
\end{equation}
We have
\begin{align*}
\left|\left\|u-x_n\right\|^2-\left\|u-y_n^i\right\|^2\right|&=\left|\left\|u-x_n\right\|-\left\|u-y_n^i\right\|\right|\left(\left\|u-x_n\right\|+\left\|u-y_n^i\right\|\right)\\ 
& \le \left\|x_n-y_n^i\right\|\left(\left\|u-x_n\right\|+\left\|u-y_n^i\right\|\right).
\end{align*}
By the boundedness of $\left\{x_n\right\},\left\{y_n^i\right\}$ and $(\ref{eq:2.44})$, we obtain
\begin{equation}\label{eq:2.49*}
\left\|u-x_n\right\|^2-\left\|u-y_n^i\right\|^2 \to 0.
\end{equation}
The last relation and $(\ref{eq:2.50})$ imply that 
\begin{equation}\label{eq:2.51}
\lim_{n\to\infty}\left\|A_{k_n}\bar{z}_n-A_{k_n}u\right\|=0.
\end{equation}
From $(\ref{eq:2.47*})$ and $(\ref{eq:2.49})$,  we obtain
\begin{align*}
\left\|u-y_n^i\right\|^2&\le\left\|u-\bar{u}_n\right\|^2\notag\\ 
& \le\left(||u-x_n||^2-||\bar{u}_n-\bar{z}_n||^2\right)+2\lambda||\bar{u}_n-\bar{z}_n||||A_{k_n}\bar{z}_n-A_{k_n}u||.
\end{align*}
Therefore,
\begin{equation}\label{eq:2.51*}
||\bar{u}_n-\bar{z}_n||^2\le\left(\left\|u-x_n\right\|^2-\left\|u-y_n^i\right\|^2\right)+2\lambda||\bar{u}_n-x_n||||A_{k_n}x_n-A_{k_n}u||.
\end{equation}
From $(\ref{eq:2.49*}), (\ref{eq:2.51}),(\ref{eq:2.51*})$ and $0<\lambda<2\alpha$, we get
\begin{equation}\label{eq:2.52}
\lim_{n\to\infty}\left\|\bar{z}_n-\bar{u}_n\right\|=0.
\end{equation}
Since $||x_n-\bar{z}_n||\to 0$ and $||x_n-\bar{u}_n||\le||x_n-\bar{z}_n||+||\bar{z}_n-\bar{u}_n||$,
\begin{equation*}
\lim_{n\to\infty}\left\|x_n-\bar{u}_n\right\|=0.
\end{equation*}
This together with $(\ref{eq:2.44}), (\ref{eq:2.45})$ implies that 
\begin{equation}\label{eq:2.53}
\lim_{n\to\infty}\left\|x_n-S_ix_n\right\|=0
\end{equation}
for all $i=1,2,\ldots,N$. By the boundedness of $\left\{x_n\right\}$, there exists a subsequence $\left\{x_m\right\}$ of $\left\{x_n\right\}$ converging weakly to $\widehat{x}\in C$. From $(\ref{eq:2.53})$ and Lemma $\ref{lem.demiclose}$, $\widehat{x}\in F(S_i)$ for all $i=1,2,\ldots,N$. Hence, $\widehat{x}\in \bigcap_{i=1}^N F(T_i)$.\\
\textbf{Step 5.} Now we show that $\widehat{x}\in \bigcap_{k=1}^MVI(A_k,C)$. Indeed, we have that
\begin{equation*}\label{eq:2.53t1}
||u_m^k-\bar{z}_m||\le||u_m^k-x_m||+||x_m-\bar{z}_m||.
\end{equation*}
Therefore, $||u_m^k-\bar{z}_m||\to 0$ as $m\to\infty$. Note that, we also have  $u_m^k\rightharpoonup \widehat{x}$ and $\bar{z}_m\rightharpoonup \widehat{x}$ as $m\to\infty$. We have
\begin{align}
\left\|\bar{z}_m-P_C(I-\lambda A_k)\widehat{x}\right\|^2=\left\|\bar{z}_m-\widehat{x}\right\|^2+&2\left\langle \bar{z}_m-\widehat{x},\widehat{x}-P_C(I-\lambda A_k)\widehat{x}\right\rangle \notag \\ 
&+\left\|\widehat{x}-P_C(I-\lambda A_k)\widehat{x}\right\|^2. \label{eq:2.54}
\end{align}
Moreover, from $u^k_m=P_C(I-\lambda A_k)\bar{z}_m$ and the nonexpansiveness of $P_C(I-\lambda A_k)$, one has
\begin{align}
\left\|\bar{z}_m-P_C(I-\lambda A_k)\widehat{x}\right\|^2&\le \left(\left\|\bar{z}_m-u_m^k\right\|+\left\|P_C(I-\lambda A_k)\bar{z}_m-P_C(I-\lambda A_k)\widehat{x}\right\|\right)^2 \notag\\
&\le \left(\left\|\bar{z}_m-u^k_m\right\|+\left\|\bar{z}_m-\widehat{x}\right\|\right)^2.\label{eq:2.55}
\end{align}
From $(\ref{eq:2.54}),(\ref{eq:2.55})$ we get
\begin{align*}
\left\|\widehat{x}-P_C(I-\lambda A_k)\widehat{x}\right\|^2\le \left\|\bar{z}_m-u^k_m\right\|^2&+2\left\|\bar{z}_m-u^k_m\right\|\left\|\bar{z}_m-\widehat{x}\right\|\\
&-2\left\langle \bar{z}_m-\widehat{x},\widehat{x}-P_C(I-\lambda A_k)\widehat{x}\right\rangle.
\end{align*}
Letting $m\to\infty$, we obtain 
\begin{equation*}
\widehat{x}=P_C(I-\lambda A_k)\widehat{x}.
\end{equation*}
By $(\ref{eq:EquivalentVIP})$, $\widehat{x}\in VI(A_k,C)$ for all $k=1,2,\ldots,M$.\\
\textbf{Step 6.} We show that $\widehat{x}\in \bigcap_{l=1}^K EP(f_l)$.\\
Note that $\lim_{n\to\infty}\left\|z_m^l-x_m\right\|=0$. This together $r_m\ge d>0$ implies that
\begin{equation}\label{eq:2.12}
\lim_{m\to\infty}\frac{\left\|z_m^l-x_m\right\|}{r_m}=0.
\end{equation}
We have that $z_m^{l}=T_{r_m}^{f_l}x_m$, i.e.,
\begin{equation}\label{eq:2.13}
f_l(z_m^l,y)+\frac{1}{r_m}\left\langle y-z_m^l,z_m^l-x_m\right\rangle \ge 0 \quad \forall y\in C.
\end{equation}
From $(\ref{eq:2.13})$ and $(A2)$, we get
\begin{equation}\label{eq:2.14}
\frac{1}{r_m}\left\langle y-z_m^l,z_m^l-x_m\right\rangle \ge -f_l(z_m^l,y)\ge f_k(y,z_m^l) \quad \forall y\in C.
\end{equation}
Taking $m\to\infty$, by $(\ref{eq:2.12}),(\ref{eq:2.14})$ and $(A4)$, we obtain
\begin{equation}\label{eq:2.15}
f_l(y,\widehat{x})\le 0,\, \forall y\in C.
\end{equation}
For $0<t\le 1$ and $y\in C$, putting $y_t=ty+(1-t)\widehat{x}$. Since $y\in C$ and $\widehat{x}\in C$, $y_t \in C$. Hence, for small sufficient $t$, from $(A1),(A3)$ and $(\ref{eq:2.15})$, we have that
\begin{equation*}
f_l(y_t,\widehat{x})=f_l(ty+(1-t)\widehat{x},\widehat{x})\le 0.
\end{equation*}
By $(A1),(A4)$, we have that
\begin{align*}
0&=f_l(y_t,y_t)\\ 
&=f_l(y_t,ty+(1-t)\widehat{x}) \\
&\le tf_l(y_t,y)+(1-t)f(y_t,\widehat{x})\\
&\le tf_l(y_t,y).
\end{align*}
Dividing both sides of the last inequality by $t>0$, we obtain $f_l(y_t,y)\ge 0$ for all $y\in C$, i.e., 
$$ f_l(ty+(1-t)\widehat{x},y)\ge 0,\, \forall y\in C. $$
Taking $t\to 0^+$, from $(A3)$, we get $f_l(\widehat{x},y)\ge 0,\, \forall y\in C$ and $l=1,2,\ldots,K$, i.e, $\widehat{x}\in \cap_{l=1}^K EP(f_l)$. Therefore, $\widehat{x}\in F.$\\
\textbf{Step 7.} We show that $x_n\to P_{F}x_0$. Setting $w=P_{F}x_0$. From  $(\ref{eq:2.37*})$, we get
\begin{equation*}
\left\|x_m-x_0\right\|\le \left\|w-x_0\right\|.
\end{equation*}
By the lower weak continuity of $\left\|.\right\|$ we have
\begin{equation*}
\left\|\widehat{x}-x_0\right\|\le \lim_{m\to\infty}\inf \left\|x_m-x_0\right\|\le \lim_{m\to\infty}\sup \left\|x_m-x_0\right\|\le \left\|w-x_0\right\|.
\end{equation*}
By the definition of $w$, $\widehat{x}=w$ and $\lim_{m\to\infty}\left\|x_m-x_0\right\|=\left\|\widehat{x}-x_0\right\|$. This implies that
\begin{equation*}
\lim_{m\to\infty}\left\|x_m\right\|=\left\|\widehat{x}\right\|.
\end{equation*}
Therefore, $\lim_{m\to\infty}x_m=\widehat{x}$. Assume that $\left\{x_k\right\}$ is an any subsequence of $\left\{x_n\right\}$. By arguing similarly to above proof, $x_k\to P_{F}x_0$ as $k\to\infty$. Hence, $x_n\to P_{F}x_0$ as $n\to\infty$. The proof of Theorem $\ref{theo:VIPandRNMs1}$ is complete.
\end{proof}
Now, we consider the ill - posed system of the operator equations
\begin{equation}\label{eq:2.62}
A_i(x)=0,\,,x\in H,\,i=1,2,\ldots, N.
\end{equation}
where $A_i:H \to H$ are possibly nonlinear operators on $H$. Let $S$ denote by the set of solutions of the system $(\ref{eq:2.62})$. An element $x^\dagger$ is called $x_0$ - minimize norm solution of the system $(\ref{eq:2.62})$ if $x^\dagger \in S$ and satisfies
\begin{equation*}\label{eq:2.63}
\left\|x^\dagger -x_0 \right\|=\min\left\{\left\|z-x_0\right\|:z\in S\right\}.
\end{equation*}
If $x_0=0$ then $x^\dagger$ is said simply to be the minimize norm solution. Several sequential and parallel iterative regularization methods \cite{ABH2014,AC2013,BK2006,CHLS2008,HLS2007} have been proposed for finding a solution of the system $(\ref{eq:2.62})$. Using Theorem $\ref{theo:VIPandRNMs1}$, we also obtain the following result:
\begin{corollary}\label{cor.3}
Let $A_i:H\to H,i=1,2,\ldots, N$ be a finite family of $\alpha$ - inverse strongly monotone mappings with the set of solutions $S$ being nonempty. The sequence $\left\{x_n\right\}$ is generated by the following manner: 
\begin{equation*}\label{eq:corollary3}
\begin{cases}
&x_0\in H,\\
&i_n:=\arg\max\left\{\left\|A_i x_n \right\|:i=1,\ldots,N\right\}, \bar{A}_n:=A_{i_n}\\
 &C_n=\left\{v\in H:\left\langle v,\bar{A}_nx_n\right\rangle\le \left\langle x_n-\mu \bar{A}_nx_n,\bar{A}_nx_n\right\rangle\right\},\\
&Q_n=\left\{v\in H: \left\langle v, x_0-x_n\right\rangle \le \left\langle x_n,x_0-x_n\right\rangle\right\},\\
&x_{n+1}=P_{C_n\cap Q_n}x_0,n\ge 0.
\end{cases}
\end{equation*}
where $\mu \in (0,\alpha)$. Then $\left\{x_n\right\}$ converges strongly to the $x_0$ - minimize norm solution $x^\dagger$ of the system $(\ref{eq:2.62})$.
\end{corollary}
\begin{proof}
Putting $C=H$, $\lambda=2\mu$, $\alpha_n=0$ for all $n\ge 0$, $S_i=I$, $f_l(x,y)=0$. Using Theorem $\ref{theo:VIPandRNMs1}$, we obtain the desired result. 
\end{proof}
Next, deals with the problem finding a common element of the set of solutions of a system of variational inequalities for $\alpha$ - inverse strongly monotone operators $\left\{A_k\right\}_{k=1}^M$ and the set of fixed points of a finite family of nonexpansive mappings $\left\{S_i\right\}_{i=1}^N$. One can employ the method $(\ref{eq:Algorithm1})$ to find this common element. We obtain the following result:
\begin{corollary}\label{cor:Algorithm3}
Let $\left\{A_k\right\}_{k=1}^M:C\to H$ be a finite family of $\alpha$ - inverse strongly monotone operators, $\left\{S_i\right\}_{i=1}^N:C\to C$ be a finite family of nonexpansive mappings. Assume that the set $F=\left(\cap_{i=1}^N F(S_i)\right)\bigcap \left(\cap_{k=1}^M VI(A_k,C)\right)$ is nonempty. Let $\left\{x_n\right\}$ be the sequence generated by the following manner: 
\begin{equation}\label{eq:Algorithm3}
\begin{cases}
&x_0\in H, C_0=Q_0=C,\\ 
&z_n=P_Cx_n,\\
&u_n^k=P_C (z_n-\lambda A_kz_n),k=1,\ldots, M,\\
& k_n:=\arg\max\left\{\left\|u_n^k -x_n \right\|:k=1,\ldots,M\right\}, \bar{u}_n:=u_n^{k_n},\\
&y_n^i=\alpha_n \bar{u}_n+(1-\alpha_n)S_i \bar{u}_n, i=1,\ldots, N,\\
& i_n:=\arg\max\left\{\left\|y_n^i -x_n \right\|:i=1,\ldots,N\right\}, \bar{y}_n:=y_n^{i_n},\\
 &C_n=\left\{v\in H:\left\|v-\bar{y}_n\right\| \le\left\|v-z_n\right\|\le \left\|v-x_n\right\|\right\},\\
&Q_n=\left\{v\in H: \left\langle x_0-x_n,x_n-v\right\rangle \ge 0\right\},\\
&x_{n+1}=P_{C_n\cap Q_n}x_0,n\ge 0.
\end{cases}
\end{equation}
where, $\lambda\in (0;2\alpha)$ and $\left\{\alpha_n\right\}\subset [0,1], \lim\sup_{n\to \infty}\alpha_n <1$. Then the sequence $\left\{x_n\right\}$ converges strongly to $P_{F}x_0$.
\end{corollary}
\begin{proof}
Putting $f_l(x,y)=0$ for all $l=1,2,\ldots, K$ and $r_n=1$. Then $T_{r_n}^{f_l}x=P_Cx$ for all $x\in H$. The proof of  Corollary $\ref{cor:Algorithm3}$ follows from Theorem $\ref{theo:VIPandRNMs1}$.
\end{proof}
However, the subset $C_{n}$ in the method $(\ref{eq:Algorithm3})$ is complex. 
 Moreover, the projection $P_{C_n\cap Q_n}x_0$ in each iterative step, in general, is difficult to find it. One assumes that $P_Cx$ can be calculated easily \cite{BBL1997,SS2000}. To overcome the complexity caused by $C_n$ and $P_{C_n\cap Q_n}$, we propose the following parallel modified algorithm:
\begin{algorithm}\label{algorithm3}
Let $x_0\in H$ be an arbitrary chosen element , $\left\{\alpha_n \right\}$ be in $[0,1]$, and $\lambda \in (0;2\alpha)$. Assume that $x_n$ is known for some $n\ge 0$.\\
\textbf{Step 1.} Calculate $z_n=P_C(x_n)$.\\
\textbf{Step 2.} Calculate the intermediate approximations $u_n^k$ in parallel
\begin{equation*}\label{eq:2.54t}
u_n^k=P_C (z_n-\lambda A_k(z_n)), k=1,2,\ldots,M.
\end{equation*}
\textbf{Step 3.} Find $k_n=\arg\max\left\{\left\|u_n^k -x_n \right\|:k=1,\ldots,M\right\}$. Put $\bar{u}_n:=u_n^{k_n}$.\\
\textbf{Step 4.} Calculate the intermediate approximations $y_n^i$ in parallel
\begin{equation*}\label{eq:2.55t}
y_n^i=\alpha_n \bar{u}_n+(1-\alpha_n)S_i \bar{u}_n,i=1,2,\ldots,N.
\end{equation*}
\textbf{Step 5.} Find $i_n=\arg\max\left\{\left\|y_n^i -x_n \right\|:i=1,\ldots,N\right\}$. Put $\bar{y}_n:=y_n^{i_n}$.\\
\textbf{Step 6.} If $||\bar{y}_n-x_n||=0$ then stop. Else, move to \textbf{Step 7}.\\
\textbf{Step 7.} Define
\begin{align*}
&C_n=\left\{v\in H:\left\|v-\bar{y}_n\right\| \le \left\|v-x_n\right\|\right\},\\
&Q_n=\left\{v\in H: \left\langle x_0-x_n,x_n-v\right\rangle \ge 0\right\}.
\end{align*}
\textbf{Step 8.} Perform 
\begin{equation*}
x_{n+1}=P_{C_n \cap Q_n}x_0.
\end{equation*}
\textbf{Step 9.} If $x_{n+1}=x_n$ then stop. Else, set $n:=n+1$ and return \textbf{Step 1}.
\end{algorithm}
% \begin{algorithm}\label{algorithm3}
% Let $x_0\in H$ be a arbitrarily chosen element , $\left\{\alpha_n \right\}$ be in $[0,1]$, and $\lambda \in (0;2c)$. Assume that $x_n$ is known for some $n\ge 0$.
% \begin{enumerate}
% \item [$Step\,\,1.$]  Calculate $z_n=P_C(x_n)$.
% \item [$Step\,2.$] Calculate the intermediate approximations $u_n^k$ in parallel
% \begin{equation}\label{eq:2.54t}
% u_n^k=P_C (z_n-\lambda A_k(z_n)), k=1,2,\ldots,M.
% \end{equation}
% \item [$Step\, 3.$] Find $k_n=\arg\max\left\{\left\|u_n^k -x_n \right\|:k=1,\ldots,M\right\}$. Put $\bar{u}_n:=u_n^{k_n}$.
% \item [$Step\,4.$] Calculate the intermediate approximations $y_n^i$ in parallel
% \begin{equation}\label{eq:2.55t}
% y_n^i=\alpha_n \bar{u}_n+(1-\alpha_n)T_i \bar{u}_n,i=1,2,\ldots,N.
% \end{equation}

Clearly, in every iterative step of Algorithm $\ref{algorithm3}$, $C_n$ and $Q_n$ are either $H$ or the half spaces. Therefore, by calculating similarly in \cite{SS2000}, we can obtain $x_{n+1}=P_{C_n\cap Q_n}x_0$ easily. Indeed, we see that $\left\|v-\bar{y}_n\right\| \le \left\|v-x_n\right\|$ is equivalent to 
\begin{equation*}
\left\langle v-\frac{x_n+\bar{y}_n}{2},x_n-\bar{y}_n\right\rangle\le 0.
\end{equation*}
Therefore, we obtain that \cite[Algorithm 1]{SS2000}
\begin{equation}\label{eq:2.56t}
x_{n+1}:=P_{C_n}x_0=x_0-\frac{\left\langle x_n -\bar{y}_n,x_0-\frac{(x_n+\bar{y}_n)}{2}\right\rangle}{||x_n -\bar{y}_n||^2}\left(x_n -\bar{y}_n\right),
\end{equation}
if $P_{C_n}x_0\in Q_n$. Else
\begin{equation}\label{eq:2.57t}
x_{n+1}=P_{C_n \cap Q_n}x_0:=x_0+\lambda_1(x_n -\bar{y}_n)+\lambda_2(x_0-x_n),
\end{equation}
where $\lambda_1,\lambda_2$ is the solution of the system of two linear equations
\begin{equation*}
\begin{cases}
&\lambda_1||x_n -\bar{y}_n||^2+\lambda_2\left\langle x_n -\bar{y}_n,x_0-x_n\right\rangle =-\left\langle x_0-\frac{x_n+\bar{y}_n}{2},x_n -\bar{y}_n\right\rangle\\
&\lambda_1\left\langle x_n -\bar{y}_n,x_0-x_n\right\rangle+\lambda_2||x_0 -x_n||^2=-||x_0 -x_n||^2.
\end{cases}
\end{equation*}
%%%%%%%%%%%%%%%%%%%%%%%%%%%%%%%%%%%%%%%%%%%%%%%%%%%%%
\begin{theorem}\label{theo:VIPandRNMs3}
Let $\left\{A_k\right\}_{k=1}^M:C\to H$ be a finite family of $\alpha$ - inverse strongly monotone operators and $\left\{S_i\right\}_{i=1}^N:C\to C$ be a finite family of nonexpansive mappings such that $F=\left(\cap_{i=1}^N F(S_i)\right)\bigcap \left(\cap_{k=1}^M VI(A_k,C)\right)\ne \O$. Assume that the sequence $\left\{\alpha_n\right\}\subset [0,1]$ satisfies $\lim\sup_{n\to \infty}\alpha_n <1$. Then the sequence $\left\{x_n\right\}$ generated by Algorithm $\ref{algorithm3}$ converges strongly to $P_{F}x_0$.
\end{theorem}
%========================================================================================
\begin{proof}
By arguing similarly to the proof of Theorem $\ref{theo:VIPandRNMs1}$ we obtain $F, C_n, Q_n$ are closed convex subsets of $C$. Now, we show that $F \subset C_n \cap Q_n$. For every $u\in F$, by the convexity of $\left\|.\right\|^2$ and the nonexpansiveness of $S_{i_n}$, we obtain
\begin{align*}
\left\|u-\bar{y}_n\right\|^2=&\left\|u-\alpha_n \bar{u}_n-(1-\alpha_n)S_{i_n} \bar{u}_n\right\|^2\notag \\
=& \left\|u\right\|^2-2\alpha_n\left\langle u,\bar{u}_n\right\rangle -2(1-\alpha_n)\left\langle u,S_{i_n} \bar{u}_n\right\rangle\notag\\
&+\left\|\alpha_n \bar{u}_n+(1-\alpha_n)S_{i_n} \bar{u}_n \right\|^2\notag\\
\le& \left\|u\right\|^2-2\alpha_n\left\langle u,\bar{u}_n\right\rangle -2(1-\alpha_n)\left\langle u,S_{i_n} \bar{u}_n\right\rangle+\alpha_n\left\|\bar{u}_n\right\|^2\notag\\
&+(1-\alpha_n)\left\|S_{i_n} \bar{u}_n\right\|^2 \notag\\
=&\alpha_n\left\|u-\bar{u}_n\right\|^2+(1-\alpha_n)\left\|u-S_{i_n} \bar{u}_n\right\|^2\notag\\
\le& \alpha_n\left\|u-\bar{u}_n\right\|^2+(1-\alpha_n)\left\|u-\bar{u}_n\right\|^2 \notag\\
=&\left\|u-\bar{u}_n\right\|^2\label{eq:2.54t}
\end{align*}
From the definition of $\bar{u}_n$, $(\ref{eq:EquivalentVIP})$ and the nonexpansiveness of $P_C(I-\lambda A_{k_n})$ and $P_C$, we have
\begin{align*}
\left\|u-\bar{u}_n\right\|&=\left\|P_C (I-\lambda A_{k_n})u-P_C (I-\lambda A_{k_n})z_n\right\|\notag\\ 
&\le \left\|u-z_n\right\|=\left\|P_Cu-P_Cx_n\right\|\notag\\
&\le \left\|u-x_n\right\|.
\end{align*}
Therefore,
\begin{equation*}\label{eq:2.55t}
\left\|u-\bar{y}_n\right\|\le\left\|u-x_n\right\|. 
\end{equation*}
This implies that $F \subset C_n$ for all $n\ge 0$. By the induction, we obtain that $F \subset C_n\cap Q_n$ for all $n\ge 0$. By arguing similarly to the proof of Theorem $\ref{theo:VIPandRNMs1}$ we obtain the sequences $\left\{x_n\right\}$, $\left\{y_n^i\right\}$, $\left\{u_n\right\}$, $\left\{T_i u_n\right\}$ are bounded and
\begin{equation}\label{eq:2.56}
\begin{cases}
&\lim_{n\to\infty}\left\|x_{n+1}-x_n \right\|=0,\\
&\lim_{n\to\infty}\left\|x_{n+1}-\bar{y}_n \right\|=0,\\
&\lim_{n\to\infty}\left\|x_{n}-y_n^i\right\|=0,\,\forall i=1,2,\ldots,N.
\end{cases}
\end{equation}
By $\bar{u}_n, T_i\bar{u}_n\in C$ and the convexity of $C$, $y_n^i\in C$. Hence $\left\|z_{n}-y_n^i\right\|=\left\|P_Cx_{n}-P_Cy_n^i\right\|\le \left\|x_{n}-y_n^i\right\|\to 0$. So, $\left\|x_{n}-z_n\right\|\le \left\|x_{n}-y_n^i\right\|+\left\|y_n^i-z_n\right\|\to 0.$ 
We have 
\begin{align*}
\left\|z_{n}-y_n^i\right\|&=\left\|\alpha_n(z_n-\bar{u}_n)+(1-\alpha_n)(z_n-T_i\bar{u}_n)\right\|\notag\\
&\ge (1-\alpha_n)\left\|z_n-T_i\bar{u}_n\right\|-\alpha_n\left\|z_{n}-\bar{u}_n\right\|.\label{eq:2.57}
\end{align*}
Therefore,
\begin{equation*}\label{eq:2.58}
\left\|z_n-T_i\bar{u}_n\right\|\le \frac{1}{1-\alpha_n}\left\|z_{n}-y_n^i\right\|+\frac{\alpha_n}{1-\alpha_n}\left\|z_{n}-\bar{u}_n\right\|.
\end{equation*}
This together with the nonexpansiveness of $T_i$ implies that
\begin{align}
\left\|z_n-T_iz_n\right\|&\le\left\|z_{n}-T_i\bar{u}_n\right\|+\left\|T_i\bar{u}_n-T_ix_n\right\|\notag\\ 
&\le \left\|z_{n}-T_i\bar{u}_n\right\|+\left\|\bar{u}_n-x_n\right\|\notag\\
&\le \frac{1}{1-\alpha_n}\left\|z_{n}-y_n^i\right\|+\frac{\alpha_n}{1-\alpha_n}\left\|z_{n}-\bar{u}_n\right\|+\left\|\bar{u}_n-z_n\right\|+\left\|z_{n}-x_n\right\|\notag\\
&\le \frac{1}{1-\alpha_n}\left\|z_{n}-y_n^i\right\|+\frac{1}{1-\alpha_n}\left\|z_{n}-\bar{u}_n\right\|+\left\|z_{n}-x_n\right\|.\label{eq:2.59}
\end{align}
By arguing similarly to $\left(\ref{eq:2.52}\right)$ we obtain 
\begin{equation}\label{eq:2.60}
\lim_{n\to\infty}\left\|z_{n}-\bar{u}_n \right\|=0.
\end{equation}
From $(\ref{eq:2.59}),(\ref{eq:2.60})$ and $\lim_{n\to\infty}\left\|z_{n}-y_n^i \right\|=\lim_{n\to\infty}\left\|z_{n}-x_n \right\|=0$ we get
\begin{equation}\label{eq:2.61}
\lim_{n\to\infty}\left\|z_{n}-T_iz_n \right\|=0.
\end{equation}
Repeating Steps 5, 6, 7 in the proof of  Theorem $\ref{theo:VIPandRNMs1}$ we get $\lim_{n\to\infty}z_{n}=P_{F}x_0.$ By $\lim_{n\to\infty}\left\|z_{n}-x_n \right\|=0$, $\lim_{n\to\infty}x_{n}=P_{F}x_0.$ The proof of Theorem $\ref{theo:VIPandRNMs3}$ is complete.
\end{proof}
%=========================================================================================
Using Theorem $\ref{theo:VIPandRNMs3}$, one gets the following result which was obtained in \cite{AC2013}.
\begin{corollary}\label{lem.5}\cite{AC2013}
Let $\left\{S_i\right\}_{i=1}^N:C\to C$ be a finite family of nonexpansive mappings with $F=\bigcap_{i=1}^NF(S_i)\ne \O$. Let $\left\{x_n\right\}$ be the sequence generated by the following algorithm:
\begin{equation*}\label{eq:5}
\begin{cases}
&x_0\in H,\\ 
&z_n=P_C (x_n),\\
&y_n^i=\alpha_n z_n+(1-\alpha_n)S_i z_n, i=1,\ldots, N,\\
& i_n:=\arg\max\left\{\left\|y_n^i -x_n \right\|:i=1,\ldots,N\right\}, \bar{y}_n:=y_n^{i_n},\\
 &C_n=\left\{v\in H:\left\|v-\bar{y}_n\right\| \le \left\|v-x_n\right\|\right\},\\
&Q_n=\left\{v\in H: \left\langle x_0-x_n,x_n-v\right\rangle \ge 0\right\},\\
&x_{n+1}=P_{C_n\cap Q_n}x_0,n\ge 0,
\end{cases}
\end{equation*}
where the sequence $\left\{\alpha_n\right\}\subset [0,1]$ satisfies $\lim\sup_{n\to \infty}\alpha_n <1$. Then the sequence $\left\{x_n\right\}$ converges strongly to $P_{F}x_0$.
\end{corollary}
\begin{proof}
Putting $A(x)=0$ for all $x\in H$. The proof of Corollary $\ref{lem.5}$ follows immediately from Theorem $\ref{theo:VIPandRNMs3}$.
\end{proof}

%=========================================================================================
\section*{Acknowledgments} The author wishes to thank Prof.Dsc Pham Ky Anh, Department of Mathematics, Hanoi University of Science, VNU for his hints concerning this paper. The author also thanks the reviewers for their valuable comments and suggestions which improved this paper.

\end{document}